\documentclass[11pt, twoside]{article}
\usepackage{amssymb, amsmath, amsthm}
\usepackage[left=25mm, right=25mm, bottom=27mm]{geometry}
\usepackage{inputenc}
\usepackage{fancyhdr}
\usepackage{tikz}
\usetikzlibrary{positioning}
\usepackage{titling}
\usepackage{hyperref}
\predate{}
\postdate{}
\usepackage{lipsum}
\newtheorem{theorem}{Theorem}
\newtheorem{lemma}[theorem]{Lemma}

\newtheorem{conjecture}[theorem]{Conjecture}
\usepackage{titling}
\usepackage{xcolor}
\usepackage{placeins}
\usepackage{tikz}
\usepackage{verbatim}
\usepackage{titling}
\usepackage[T1]{fontenc}
\usepackage{currvita}
\usepackage{bbm}
\usepackage{enumitem}
\newtheorem{corollary}[theorem]{Corollary}

\newtheorem{claim5}{Claim}

\newtheorem{claim7}{Claim}
\newtheorem{claim8}{Claim}
\newtheorem{claim9}{Claim}

\newtheorem{lemma1}{Lemma}
\newtheorem{case1}{Case}

\theoremstyle{definition}

\theoremstyle{remark}

\newcommand{\cP}{\mathcal{P}}

\newcommand{\cF}{\mathcal{F}}

\newcommand{\cM}{\mathcal M}

\begin{document}
\newcommand{\Addresses}{
\bigskip
\footnotesize

\medskip

\noindent Maria-Romina~Ivan, \textsc{Department of Pure Mathematics and Mathematical Statistics, Centre for Mathematical Sciences, Wilberforce Road, Cambridge, CB3 0WB, UK,} and\\\textsc{Department of Mathematics, Stanford University, 450 Jane Stanford Way, CA 94304, USA.}\par\noindent\nopagebreak\textit{Email addresses: }\texttt{mri25@dpmms.cam.ac.uk, m.r.ivan@stanford.edu}

\medskip

\noindent Sean~Jaffe, \textsc{Trinity College, University of Cambridge, CB2 1TQ, UK.}\par\noindent\nopagebreak\textit{Email address: }\texttt{scj47@cam.ac.uk}}

\pagestyle{fancy}
\fancyhf{}
\fancyhead [LE, RO] {\thepage}
\fancyhead [CE] {MARIA-ROMINA IVAN AND SEAN JAFFE}
\fancyhead [CO] {SATURATION FOR SUMS OF POSETS AND ANTICHAINS}
\renewcommand{\headrulewidth}{0pt}
\renewcommand{\l}{\rule{6em}{1pt}\ }
\title{\Large{\textbf{SATURATION FOR SUMS OF POSETS AND ANTICHAINS}}}
\author{MARIA-ROMINA IVAN AND SEAN JAFFE}
\date{ }
\maketitle
\begin{abstract}
Given a finite poset $\mathcal P$, we say that a family $\mathcal F$ of subsets of $[n]$ is $\mathcal P$-saturated if $\mathcal F$ does not contain an induced copy of $\mathcal P$, but adding any other set to $\mathcal F$ creates an induced copy of $\mathcal P$. The saturation number of $\mathcal P$ is the size of the smallest $\mathcal P$-saturated family with ground set $[n]$.

The saturation numbers have been shown to exhibit a dichotomy: for any poset, the saturation number is either bounded, or at least $2\sqrt n$. The general conjecture is that in fact, the saturation number for any poset is either bounded, or at least linear.

The linear sum of two posets $\mathcal P_1$ and $\mathcal P_2$, dented by $\mathcal P_1*\mathcal P_2$, is defined as the poset obtained from a copy of $\mathcal P_1$ placed completely on top of a copy of $\mathcal P_2$. In this paper we show that the saturation number of $\mathcal P_1*\mathcal A_k*\mathcal P_2$ is always at least linear, for any $\mathcal P_1$, $\mathcal P_2$ and $k\geq2$, where $\mathcal A_k$ is the antichain of size $k$. This is a generalisation of the recent result that the saturation number for the diamond is linear (in that case $\mathcal P_1$ and $\mathcal P_2$ are both the single point poset, and $k=2$). We also show that, with the exception of chains which are known to have bounded saturation number, the saturation number for all complete multipartite posets is linear.
\end{abstract}
\section{Introduction}
We say that a poset $(\mathcal Q,\preceq)$ contains an \textit{induced copy} of a poset $(\mathcal P,\preceq')$ if there exists an injective order-preserving function $f:\mathcal P\rightarrow\mathcal Q$ such that $(f(\mathcal P),\preceq)$ is isomorphic to $(\mathcal P,\preceq')$. We denote by $\mathcal P([n])$ the power set of $[n]=\{1,2,\dots,n\}$. More generally, for any finite set $S$, we denote by $\mathcal P(S)$ the power set of $S$. We define the \textit{$n$-hypercube}, denoted by $Q_n$, to be the poset formed by equipping $\mathcal P([n])$ with the partial order induced by inclusion.

If $\mathcal P$ is a finite poset and $\mathcal F$ is a family of subsets of $[n]$, we say that $\mathcal F$ is $\mathcal P$-\textit{saturated} if $\mathcal F$ does not contain an induced copy of $\mathcal P$ and, for any $S\notin\mathcal F$, the family $\mathcal F\cup\{S\}$ contains an induced copy of $\mathcal P$. The smallest size of a $\mathcal P$-saturated family of subsets of $[n]$ is called the \textit{induced saturated number}, denoted by $\text{sat}^*(n,\mathcal P)$.

Poset saturation is a relatively young and rapidly growing area in combinatorics. Saturation for posets was introduced by Gerbner, Keszegh, Lemons, Palmer, P{\'a}lv{\"o}lgyi and Patk{\'o}s \cite{gerbner2013saturating}, although this was not for \textit{induced} saturation. Induced poset saturation was first introduced in 2017 by Ferrara, Kay, Kramer, Martin, Reiniger, Smith and Sullivan \cite{ferrara2017saturation}. It was soon noted that the saturation numbers for posets have a puzzling dichotomy. Keszegh, Lemons, Martin, P{\'a}lv{\"o}lgyi and Patk{\'o}s \cite{keszegh2021induced} proved that for any poset the induced saturated number is either bounded or at least $\log_2(n)$. Later, Freschi, Piga, Sharifzadeh and Treglown \cite{freschi2023induced} improved this result by replacing $\log_2 (n)$ with $2\sqrt{n}$. The most important conjecture in this area is the following.
\begin{conjecture}\label{conjecturelinear}
Let $\mathcal P$ be a finite poset. Then $\text{sat}^*(n,\mathcal P)$ is either bounded, or at least linear.
\end{conjecture}
As of now, there are very few posets for which the above conjecture has been shown to hold. Some posets for which it is known that the saturation number is linear include $\mathcal V_2$, $\Lambda_2$ \cite{ferrara2017saturation}, the butterfly \cite{ivan2020saturationbutterflyposet, keszegh2021induced}, $\mathcal{V}_3$, $\Lambda_3$ \cite{brownlie2025induced} the antichain \cite{keszegh2021induced, dhankovic2023saturation, bastide2024exact}, and most recently the famous diamond \cite{diamondlinear}.

Building on the work in \cite{diamondlinear}, we show that the linear lower bound for the diamond is a more general phenomenon. More precisely, for two finite posets $\mathcal P_1$ and $\mathcal P_2$, we define $\mathcal P_2*\mathcal P_1$ to be the poset comprised of an induced copy of $\mathcal P_2$ completely above an induced copy of $\mathcal P_1$, i.e all elements of $\mathcal P_2$ are strictly greater than all the elements of $\mathcal P_1$. For completeness, if $\mathcal P_2$ is the empty poset, then we define $\mathcal P_2*\mathcal P_1$ to be $\mathcal P_1$, and if $\mathcal P_1$ is the empty poset we define $\mathcal P_2*\mathcal P_1$ to be $\mathcal P_2$. Trivially, this operation is associative. For example, the diamond poset is in fact $\bullet*\mathcal A_2*\bullet$, where $\mathcal A_2$ is the antichain of size 2 and $\bullet$ is the poset consisting of a single element. In this paper we show that the saturation number of $\mathcal P_1*\mathcal A_k*\mathcal P_2$ is at least linear for any finite posets $\mathcal P_1$ and $\mathcal P_2$ and positive integer $k\geq2$, where $\mathcal A_k$ is the antichain of size $k$, hence generalising the result for the diamond. For a visual aid, the Hasse diagrams of $\mathcal P_1*\mathcal P_2$ and $\mathcal P_1*\mathcal A_k*\mathcal P_2$ are illustrated below.
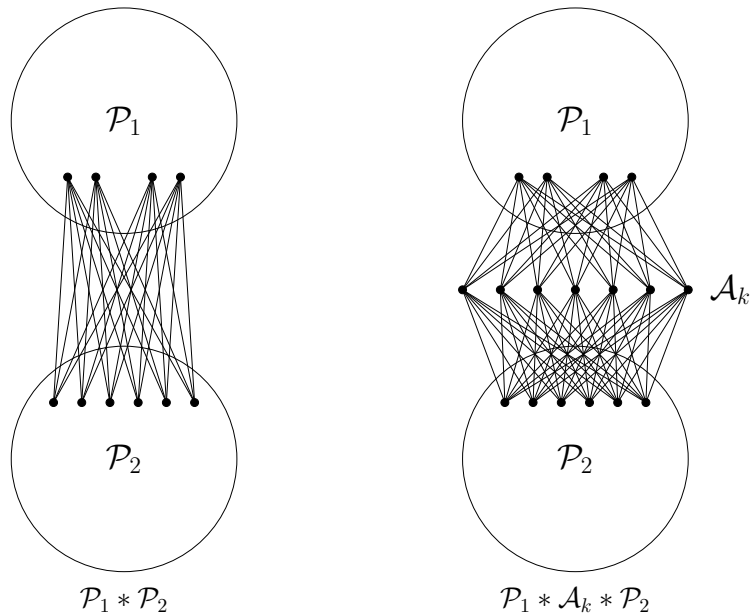
\begin{figure}[h]

    \centering
    \scalebox{0.75}{
    \begin{tikzpicture}
        \draw (0,0) circle (2cm);
        \draw (0,6) circle (2cm);

        \filldraw[black] (0.25,1) circle (2pt);
        \filldraw[black] (0.75,1) circle (2pt);
        \filldraw[black] (1.25,1) circle (2pt);
        \filldraw[black] (-0.25,1) circle (2pt);
        \filldraw[black] (-0.75,1) circle (2pt);
        \filldraw[black] (-1.25,1) circle (2pt);

        \filldraw[black] (0.5,5) circle (2pt);
        \filldraw[black] (1,5) circle (2pt);
        \filldraw[black] (-0.5,5) circle (2pt);
        \filldraw[black] (-1,5) circle (2pt);

        \draw (-1.25,1) -- (-1,5);
        \draw (-1.25,1) -- (-0.5,5);
        \draw (-1.25,1) -- (0.5,5);
        \draw (-1.25,1) -- (1,5);
        \draw (-0.75,1) -- (-1,5);
        \draw (-0.75,1) -- (-0.5,5);
        \draw (-0.75,1) -- (0.5,5);
        \draw (-0.75,1) -- (1,5);
        \draw (-0.25,1) -- (-1,5);
        \draw (-0.25,1) -- (-0.5,5);
        \draw (-0.25,1) -- (0.5,5);
        \draw (-0.25,1) -- (1,5);
        \draw (1.25,1) -- (-1,5);
        \draw (1.25,1) -- (-0.5,5);
        \draw (1.25,1) -- (0.5,5);
        \draw (1.25,1) -- (1,5);
        \draw (0.75,1) -- (-1,5);
        \draw (0.75,1) -- (-0.5,5);
        \draw (0.75,1) -- (0.5,5);
        \draw (0.75,1) -- (1,5);
        \draw (0.25,1) -- (-1,5);
        \draw (0.25,1) -- (-0.5,5);
        \draw (0.25,1) -- (0.5,5);
        \draw (0.25,1) -- (1,5);

        \node at (0,0) {\LARGE$\mathcal{P}_2$};
        \node at (0,6) {\LARGE$\mathcal{P}_1$};
        \node at (0,-2.5) {\Large$\mathcal{P}_1\ast\mathcal{P}_2$};

        \draw (8,0) circle (2cm);
        \draw (8,6) circle (2cm);

        \filldraw[black] (8.25,1) circle (2pt);
        \filldraw[black] (8.75,1) circle (2pt);
        \filldraw[black] (9.25,1) circle (2pt);
        \filldraw[black] (7.75,1) circle (2pt);
        \filldraw[black] (7.25,1) circle (2pt);
        \filldraw[black] (6.75,1) circle (2pt);

        \filldraw[black] (8.5,5) circle (2pt);
        \filldraw[black] (9,5) circle (2pt);
        \filldraw[black] (7.5,5) circle (2pt);
        \filldraw[black] (7,5) circle (2pt);

        \filldraw[black] (8,3) circle (2pt);
        \filldraw[black] (8.67,3) circle (2pt);
        \filldraw[black] (9.33,3) circle (2pt);
        \filldraw[black] (10,3) circle (2pt);
        \filldraw[black] (7.33,3) circle (2pt);
        \filldraw[black] (6.67,3) circle (2pt);
        \filldraw[black] (6,3) circle (2pt);

        \draw (6.75,1) -- (6,3);
        \draw (6.75,1) -- (6.67,3);
        \draw (6.75,1) -- (7.33,3);
        \draw (6.75,1) -- (8,3);
        \draw (6.75,1) -- (8.67,3);
        \draw (6.75,1) -- (9.33,3);
        \draw (6.75,1) -- (10,3);
        
        \draw (7.25,1) -- (6,3);
        \draw (7.25,1) -- (6.67,3);
        \draw (7.25,1) -- (7.33,3);
        \draw (7.25,1) -- (8,3);
        \draw (7.25,1) -- (8.67,3);
        \draw (7.25,1) -- (9.33,3);
        \draw (7.25,1) -- (10,3);
        
        \draw (7.75,1) -- (6,3);
        \draw (7.75,1) -- (6.67,3);
        \draw (7.75,1) -- (7.33,3);
        \draw (7.75,1) -- (8,3);
        \draw (7.75,1) -- (8.67,3);
        \draw (7.75,1) -- (9.33,3);
        \draw (7.75,1) -- (10,3);

        \draw (8.25,1) -- (6,3);
        \draw (8.25,1) -- (6.67,3);
        \draw (8.25,1) -- (7.33,3);
        \draw (8.25,1) -- (8,3);
        \draw (8.25,1) -- (8.67,3);
        \draw (8.25,1) -- (9.33,3);
        \draw (8.25,1) -- (10,3);

        \draw (8.75,1) -- (6,3);
        \draw (8.75,1) -- (6.67,3);
        \draw (8.75,1) -- (7.33,3);
        \draw (8.75,1) -- (8,3);
        \draw (8.75,1) -- (8.67,3);
        \draw (8.75,1) -- (9.33,3);
        \draw (8.75,1) -- (10,3);

        \draw (9.25,1) -- (6,3);
        \draw (9.25,1) -- (6.67,3);
        \draw (9.25,1) -- (7.33,3);
        \draw (9.25,1) -- (8,3);
        \draw (9.25,1) -- (8.67,3);
        \draw (9.25,1) -- (9.33,3);
        \draw (9.25,1) -- (10,3);

        \draw (7,5) -- (6,3);
        \draw (7,5) -- (6.67,3);
        \draw (7,5) -- (7.33,3);
        \draw (7,5) -- (8,3);
        \draw (7,5) -- (8.67,3);
        \draw (7,5) -- (9.33,3);
        \draw (7,5) -- (10,3);

        \draw (7.5,5) -- (6,3);
        \draw (7.5,5) -- (6.67,3);
        \draw (7.5,5) -- (7.33,3);
        \draw (7.5,5) -- (8,3);
        \draw (7.5,5) -- (8.67,3);
        \draw (7.5,5) -- (9.33,3);
        \draw (7.5,5) -- (10,3);

        \draw (8.5,5) -- (6,3);
        \draw (8.5,5) -- (6.67,3);
        \draw (8.5,5) -- (7.33,3);
        \draw (8.5,5) -- (8,3);
        \draw (8.5,5) -- (8.67,3);
        \draw (8.5,5) -- (9.33,3);
        \draw (8.5,5) -- (10,3);

        \draw (9,5) -- (6,3);
        \draw (9,5) -- (6.67,3);
        \draw (9,5) -- (7.33,3);
        \draw (9,5) -- (8,3);
        \draw (9,5) -- (8.67,3);
        \draw (9,5) -- (9.33,3);
        \draw (9,5) -- (10,3);

        \node at (8,0) {\LARGE$\mathcal{P}_2$};
        \node at (8,6) {\LARGE$\mathcal{P}_1$};
        \node at (10.75,3) {\LARGE$\mathcal{A}_k$};
        \node at (8,-2.5) {\Large$\mathcal{P}_1\ast\mathcal{A}_k\ast\mathcal{P}_2$};
        
    \end{tikzpicture}}
\end{figure}

The plan of the paper is as follows. In Section 2 we reduce the problem to proving linearity in the case where the top poset has a unique minimal element, and the bottom poset has a unique maximal element -- this will make the analysis to come more fluid. In Sections 3 and 4 we establish the preliminary lemmas which will enable us to apply the set-theoretic result about pairs of families of sets developed in \cite{diamondlinear} (Lemma 6), with a necessary straightforward modification, and hence prove our main result in Section 5. In Section 6 we focus on the complete multipartite posets. We show that, with the exception of a chain, which was shown in \cite{gerbner2013saturating} to have bounded saturation number, the saturation number for all multipartite posets is linear. Finally, in Section 7 we discuss briefly what happens if one replaces the antichain in the middle with a poset that has the UCTP. This is a natural question since it was shown in \cite{ferrara2017saturation} that a poset with UCTP has unbounded saturation number. However, linearity has not been established for this class of posets. 

In a poset we say that $y$ covers $x$ if $x<y$ and there is no $z$ with $x<z<y$. We say that a poset has the UCTP if, for any two elements $p_1$ and $p_2$ of the poset, if $p_2$ covers $p_1$, then there exists $p_3$, called their \textit{twin}, that is comparable with exactly one of $p_1$ and $p_2$. We mention that in the literature there is some variation when it comes to the definitions of the UCTP and that of the twin. They are equivalent as far as the results are concerned. In Section 7 we give an outline of the proof that the saturation number of $\mathcal P_1*\mathcal Q*\mathcal P_2$ is unbounded, where $\mathcal P_1$ and $\mathcal P_2$ are two finite posets, and $\mathcal Q$ is a poset with the UCTP of size at least $2$.

Throughout the paper our notation is standard. For a finite set $S$, we denote by $\binom{S}{k}$ the collection of all subsets of $S$ of size $k$. When $S=[n]$ this is also known as the $k^{\text{th}}$ layer of $Q_n$. We also denote by $\binom{S}{\leq k}$ the collection of all subsets of $S$ of size at most $k$.
\section{The setup}
Let $\mathcal P_1$ and $\mathcal P_2$ be two posets. Our goal is to show that, given any integer $k\geq2$, the saturation number of $\mathcal P_1*\mathcal A_k*\mathcal P_2$ is at least linear.

We start with a result proven in \cite{gluing}.
\begin{lemma}[Proposition 3 in \cite{gluing}]
\label{originalguinglemma}
Let $\mathcal Q_1$ and $\mathcal Q_2$ be non-empty posets such that $\mathcal Q_1$ does not have a unique minimal element and $\mathcal Q_2$ does not have a unique maximal element. Let $\mathcal F$ be a $\mathcal Q_1*\mathcal Q_2$-saturated family with ground set $[n]$. Then $\mathcal F$ is also a $\mathcal Q_1*\bullet*\mathcal Q_2$-saturated family with ground set $[n]$, where $\bullet$ represents the poset with a single element.
\end{lemma}
This gives us the following corollary. 
\begin{corollary}\label{poset-point-poset-corollary}
Let $\mathcal Q_1$ and $\mathcal Q_2$ be two posets, not both empty, such that $\mathcal Q_1$ does not have a unique minimal element and $\mathcal Q_2$ does not have a unique maximal element. Let $\mathcal F$ be a $\mathcal Q_1*\mathcal Q_2$-saturated family with ground set $[n]$. Then $\mathcal F$ is also a $\mathcal Q_1*\bullet*\mathcal Q_2$-saturated family with ground set $[n]$, where $\bullet$ represents the poset with a single element.
\end{corollary}
\begin{proof}
By symmetry, we may assume without loss of generality that $\mathcal Q_2$ is not empty. If $\mathcal Q_1$ is not empty, then we are done by Lemma~\ref{originalguinglemma}. Hence, we may assume that $\mathcal Q_1$ is empty, which means that $\mathcal F$ is $\mathcal Q_2$-saturated. Since $\mathcal Q_2$ does not have a unique maximal element, the full set $[n]$ cannot be part of any induced copy of $\mathcal Q_2$, and so $[n]\in \mathcal F$.

Since $\mathcal F$ does not contain a copy if $\mathcal Q_2$, it also does not contain a copy of $\bullet*\mathcal Q_2$. Moreover, if $X\in \mathcal P([n])\setminus \mathcal F$, then $\mathcal F\cup \{X\}$ contains a copy of $\mathcal Q_2$. By combining this copy with $[n]$, we obtain a copy of $\bullet * \mathcal Q_2$ in $\mathcal F\cup \{X\}$. Therefore, $\mathcal F$ is $\bullet * \mathcal Q_2$-saturated, as desired.
\end{proof}

Let $\mathcal P_1'$ be $\mathcal P_1$, if $\mathcal P_1$ has a unique minimal element, and $\mathcal P_1*\bullet$ otherwise. Similarly, let $\mathcal P_2'$ be $\mathcal P_2$, if $\mathcal P_2$ has a unique maximal element, and $\bullet*\mathcal P_2$ otherwise. After two straightforward applications of Corollary~\ref{poset-point-poset-corollary}, we get the following.

\begin{corollary}\label{assuminguniquemaximal}
Let $\mathcal P_1$ and $\mathcal P_2$ be two posets, and $k\geq2$. Let $\mathcal P_1'$ and $\mathcal P_2'$ be as defined above. Then $\text{sat}^*(n,\mathcal P_1* \mathcal A_k*\mathcal P_2)\geq\text{sat}^*(n,\mathcal P_1'* \mathcal A_k*\mathcal P_2')$.
\end{corollary}

This tells us that, in order to show that the saturation number of $\mathcal P_1*\mathcal A_k*\mathcal P_2$ is at least linear, it is enough to show that the saturation number of $\mathcal P_1'*\mathcal A_k*\mathcal P_2'$ is at least linear. Therefore, we may assume from now on that $\mathcal P_1$ has a unique minimal element, and $\mathcal P_2$ has a unique maximal element. 

\section{Preliminary lemmas}

Let us denote by $\widehat{\mathcal P} = \mathcal P_1* \mathcal A_k*\mathcal P_2$, and let $\mathcal F$ be a $\widehat{\mathcal P}$-saturated family.

Let $\mathcal M_0$ be the set of elements in $\mathcal F$ that are above a copy of $\mathcal P_2$, potentially equal to the unique maximal element of $\mathcal P_2$. More precisely, $\mathcal M_0 = \{A\in \mathcal F : \mathcal P(A)\cap \mathcal F \text{ contains an induced copy of }\mathcal P_2\}$. Let $\mathcal M$ be the set of minimal elements of $\mathcal M_0$. 

Similarly, we define $\mathcal N_0 = \{A\in \mathcal F: \{X\in \mathcal F: A\subseteq X\} \text{ contains an induced copy of } \mathcal P_1\}$, and $\mathcal N$ the set of maximal elements of $\mathcal N_0$.

The first helpful observation is the following lemma.
\begin{lemma}\label{firstprelim}
Suppose that $|\mathcal F|<n$, and let $i\in [n]$. Then $i\in M$ for some $M\in\mathcal M$, or there exists $Y\in \mathcal F$ with $i\notin Y$ such that $M\subsetneq Y\subsetneq N$ for some $M\in \mathcal M$ and $N\in\mathcal N$, or there exists a set $S$ such that $i\notin S$ and $S,S\cup \{i\}\in\mathcal F$.
\end{lemma}
\begin{proof}
Since $|\mathcal F|<n$, there exists $j\in [n]$ such that $[n]\setminus \{j\}\notin \mathcal F$. This means that $\mathcal F\cup \{[n]\setminus\{j\}\}$ contains a copy of $\widehat{\mathcal P}$, which must contain $[n]\setminus\{j\}$. Since $[n]$ is the only element of $\mathcal P([n])$ greater than $[n]\setminus \{j\}$, there cannot be an antichain of size $k$ above $[n]\setminus\{j\}$. This means that $[n]\setminus\{j\}$ cannot be in the $\mathcal P_2$ part of the copy of $\widehat{\mathcal P}$. Hence, there is a copy of $\mathcal P_2$ in $\mathcal F$ that is below $[n]\setminus\{j\}$. Let $A_0$ be the maximal element of such a copy. This means that $A_0\in \mathcal M_0$. Thus $\mathcal M_0$ is not empty, and so $\mathcal M$ is not empty. 

Let now $A\in \mathcal M$ with $|A|$ minimal. If $i\in A$, then we are done, so we may assume that $i\notin A$. If $A\cup\{i\}\in\mathcal F$, then we are also done, so we may also assume that $A\cup \{i\}\notin\mathcal F$. This means that $\mathcal F\cup \{A\cup \{i\}\}$ contains an induced copy of $\widehat{\mathcal P}$, which must contain $A\cup \{i\}$. Fix such a copy, which we refer to as $\widehat{\mathcal P}$ for readability purposes.
\begin{case1}
Suppose that $A\cup \{i\}$ is in the $\mathcal P_2$ part of this copy of $\widehat{\mathcal P}$.
\end{case1}
\noindent Since $A\in \mathcal M$, there exists a copy of $\mathcal P_2$ in $\mathcal F\cap \mathcal P(A)$. This copy must be strictly below $A\cup \{i\}$. However, we assumed that $A\cup\{i\}$ is in the $\mathcal P_2$ part of our copy of $\widehat{\mathcal P}$, which means that it is below a copy of $\mathcal P_1*\mathcal A_k$ in $\mathcal F$. Combining these two observations, we obtain a copy of $\mathcal P_1*\mathcal A_k*\mathcal P_2$ in $\mathcal F$, a contradiction.
\begin{case1}
Suppose that $A\cup \{i\}$ is in the $\mathcal P_1$ part of this copy of $\widehat{\mathcal P}$.
\end{case1}
\noindent Let $X$ be the unique maximal element of the $\mathcal P_2$ part of our copy of $\widehat{\mathcal P}$, and let $T$ be a member of its $\mathcal A_k$ part. By assumption, we have that $X\subsetneq T\subsetneq A\cup \{i\}$, and so $|X|<|T|<|A\cup \{i\}|$, which consequently implies that $|X|<|A|$.
    
However, by construction, the $\mathcal P_2$ part is contained in $\mathcal P(X)\cap \mathcal F$, which implies that $X\in \mathcal M_0$. Thus, there exists $Y\in\mathcal M$ such that $Y\subseteq X$. But then $Y\in\mathcal M$ and $|Y|\leq|X|<|A|$, which contradicts the minimality of $|A|$.

\begin{case1}
Suppose that $A\cup \{i\}$ is in the $\mathcal A_k$ part of this copy of $\widehat{\mathcal P}$.
\end{case1}
\noindent Let $X$ be the unique maximal element of the $\mathcal P_2$ part, $Y$ the unique minimal element of the $\mathcal P_1$ part, and $\{T_1,T_2,\dots,T_{k-1}\}$ the $\mathcal A_k$ part, without $A\cup\{i\}$. Since $X$ is above the $\mathcal P_2$ part, we have that $X\in \mathcal M_0$. Thus, there exists $X^*\in\mathcal M$ such that $X^*\subseteq X$. Similarly, we have $Y^*\in \mathcal N$ such that $Y\subseteq Y^*$. We then have the following diagram.
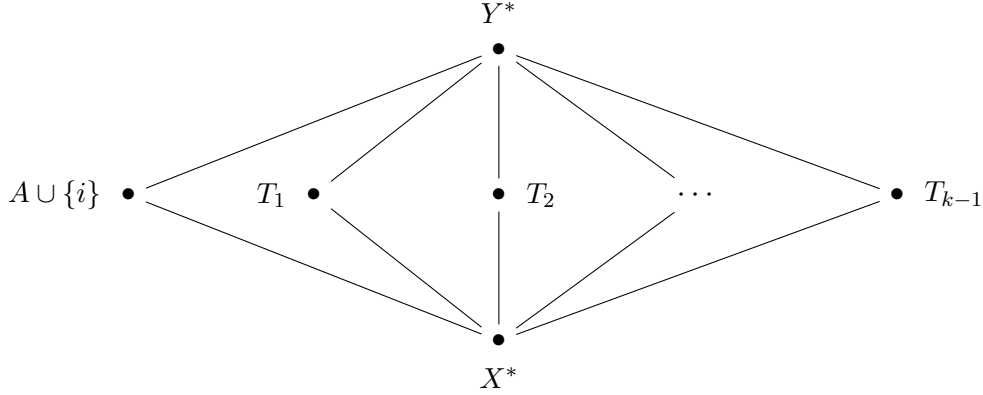
\begin{figure}[h]
\hspace{0.1cm}\\
\centering
\begin{tikzpicture}[node distance=1.5cm and 2cm, baseline=(mid.base)]
\node[label=left:$A\cup\{i\}$](u1){$\bullet$};
\node[right=of u1,label=left:$T_1$](u2){$\bullet$};
\node[right=of u2,label=right:$T_2$](u3) {$\bullet$};
\node[font=\large,right=of u3](mid){$\cdots$};
\node[right=of mid,label=right:$T_{k-1}$] (ukm1) {\(\bullet\)};
\node[above=of u3, label= above:$Y^*$] (Y) {$\bullet$};
\node[below=of u3,label=below:$X^*$] (x) {$\bullet$};
\foreach \v in {u1,u2,ukm1,u3,mid}{\draw (x) -- (\v) -- (Y);}
\end{tikzpicture}
\end{figure}
\FloatBarrier
Since $X^*\in \mathcal M$, by the minimality of $|A|$, we have that $|A|\leq|X^*|$. On the other hand, since $X^*\subsetneq A\cup \{i\}$, we have that $|X^*|\leq |A|$, which implies that $|X^*| = |A|$. 

If $X^*\neq A$, then $i\in X^*$, and we would be done as $X^*\in\mathcal M$. On the other hand, if $X^* = A$, then $A\subset T_1$ and $T_1$ and $A\cup \{i\}$ are incomparable, which means that $i\notin T_1$. Since $X^*\subsetneq T_1\subsetneq Y^*$, $X^*\in\mathcal M$, and $Y^*\in\mathcal N$, we are done since we satisfied the second condition of the statement.
\end{proof}

The next lemma is a refinement of Lemma~\ref{firstprelim}, and indeed, its proof relies on Lemma~\ref{firstprelim}. This is an intermediary step in order to reach Lemma~\ref{mainprelim}, which is the main result of this section. 

\begin{lemma}\label{secondprelim}
Suppose that $|\mathcal F|<n$, and let $i\in [n]$. Then $i\in M$ for some $M\in \mathcal M$, or $i\notin N$ for some $N\in \mathcal N$, or there exists a set $S$ such that $i\notin S$ and $S,S\cup \{i\}\in \mathcal F$.\end{lemma}
\begin{proof}
By Lemma~\ref{firstprelim}, we may assume that there exists $Y\in\mathcal F$ such that $i\notin Y$ and $M\subsetneq Y\subsetneq N$ for some $M\in\mathcal M$ and $N\in\mathcal N$, otherwise we are immediately done. Let $Y$ have maximal size with respect to the above properties. If $i\notin N$, then we are done, so we may assume that $i\in N$, which implies that $Y\cup\{i\}\subseteq N$. If $Y\cup \{i\}\in \mathcal F$, then we are done, so we may assume that $Y\cup\{i\}\notin\mathcal F$. Therefore $\mathcal F\cup \{Y\cup \{i\}\}$ contains a copy of $\widehat{\mathcal P}$ which contains $Y\cup \{i\}$ as an element. 

Suppose that $Y\cup \{i\}$ is in the $\mathcal P_2$ part of the copy of $\widehat{\mathcal P}$. Then $Y\cup \{i\}$ is below a copy of $\mathcal P_1*\mathcal A_k$ in $\mathcal F$. However, since $M\in \mathcal M$ and $M\subseteq Y\subset Y\cup\{i\}$; $M$, and consequently $Y\cup \{i\}$, is above a copy of $\mathcal P_2$ in $\mathcal F$. Combining the copy of $\mathcal P_1*\mathcal A_k$ above $Y\cup \{i\}$ and the copy of $\mathcal P_2$ below $Y\cup\{i\}$, we obtain a copy of $\widehat{\mathcal P}$ in $\mathcal F$, a contradiction.

Suppose now that $Y\cup \{i\}$ is in the $\mathcal P_1$ part of the copy of $\widehat{\mathcal P}$. Then $Y\cup \{i\}$ is above a copy of $\mathcal A_k* \mathcal P_2$ in $\mathcal F$. Since $N\in \mathcal F$ and $Y\cup \{i\}\notin \mathcal F$, we must have $Y\cup\{i\}\subsetneq N$. Since $N\in\mathcal N$, then $Y\cup\{i\}$ is strictly below a copy of $\mathcal P_1$ in $\mathcal F$. Combining this copy with the copy of $\mathcal A_k* \mathcal P_1$ below $Y\cup \{i\}$, we obtain a copy of $\widehat{\mathcal P}$ in $\mathcal F$, a contradiction.

Therefore $Y\cup \{i\}$ must be in the $\mathcal A_k$ part of the copy of $\widehat{\mathcal P}$ in $\mathcal F\cup \{Y\cup \{i\}\}$, which we call $\widehat{\mathcal P}$ for clarity. Let $E$ be the unique maximal element of the $\mathcal P_2$ part of this copy, and $F$ the unique minimal element of the $\mathcal P_1$ part of this copy. By definition we have that $E\in \mathcal M_0$ and $F\in \mathcal N_0$. Since $E^*\subseteq E$ for some $E^*\in \mathcal M$, then replacing the current copy of $\mathcal P_2$ with a copy of $\mathcal P_2$ below $E^*$ gives another copy of $\widehat{\mathcal P}$ in $\mathcal F\cup\{Y\cup\{i\}\}$, with $Y\cup\{i\}$ in its $\mathcal A_k$ part still. As such, we may assume that $E\in \cM$. Similarly, we may also assume that $F\in \mathcal N$. 

Let $T_1,T_2,\dots,T_{k-1}$ be the elements of the $\mathcal A_k$ part, other that $Y\cup\{i\}$. We have the following diagram.
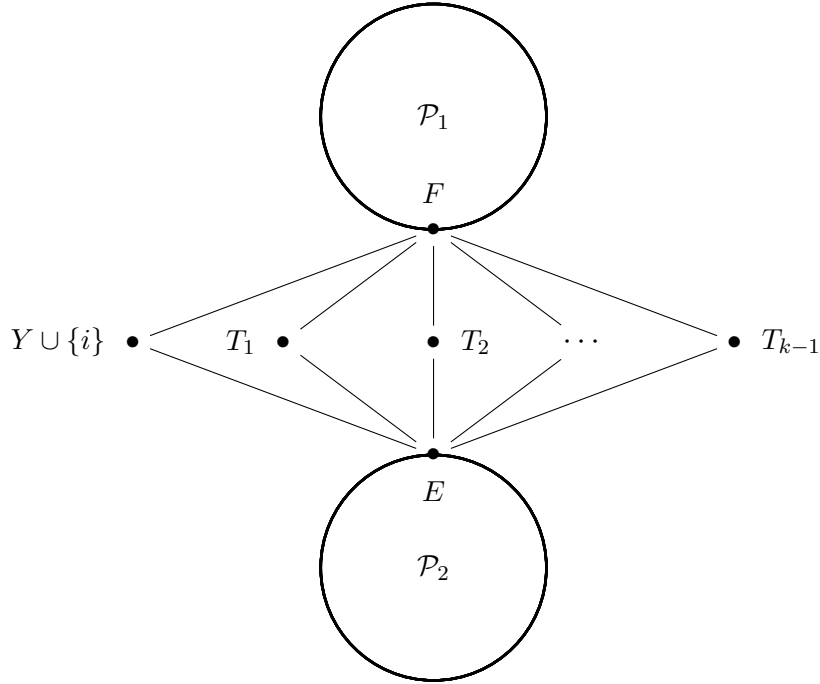
\begin{figure}[h]
\hspace{0.1cm}\\
\centering
\begin{tikzpicture}[node distance=1.5cm and 2cm, baseline=(mid.base)]
\node[label=left:$Y\cup\{i\}$] (u1) at (0,0) {$\bullet$};
\node[label=left:$T_1$] (u2) at (2,0) {$\bullet$};
\node[label=right:$T_2$] (u3) at (4,0) {$\bullet$};
\node[font=\large] (mid) at (6,0) {$\cdots$};
\node[label=right:$T_{k-1}$] (ukm1) at (8,0) {\(\bullet\)};
\node[ label= above:$F$] (Y) at (4,1.5) {$\bullet$};
\node[label=below:$E$] (x) at (4,-1.5) {$\bullet$};
\foreach \v in {u1,u2,ukm1,u3,mid}{\draw (x) -- (\v) -- (Y);
\draw[line width = 1pt] (4,3) circle(1.5cm);
\draw[line width = 1pt] (4,-3) circle(1.5cm);
}
\scoped[every node/.style=]\node at (4, -3) {$\mathcal P_2$};
\scoped[every node/.style=]\node at (4, 3) {$\mathcal P_1$};
\end{tikzpicture}
\end{figure}
\FloatBarrier
If $i\in\ E$, then we are done, so we may assume that $i\notin E$. Since $E\subset Y\cup\{i\}$, this implies that $E\subseteq Y$. Recall that we have $M\subsetneq Y\subsetneq N$ for some $M\in\mathcal M$ and $N\in\mathcal N$. Thus, $E\neq Y$, otherwise $M\subsetneq E$, contradicting the minimality of $E$. Therefore we have $E\subsetneq Y$. Note that we also have $Y\subset Y\cup\{i\}\subsetneq F$. 

Since $\mathcal F$ is $\widehat{\mathcal P}$-free, replacing $Y\cup\{i\}$ with $Y$ cannot yield a copy of $\widehat{\mathcal P}$. Therefore we must have $Y$ comparable to $T_j$ for some $j\in[k-1]$. Without loss of generality, suppose that $Y$ and $T_1$ are comparable. Since $T_1$ and $Y\cup\{i\}$ are incomparable, we cannot have $T_1\subseteq Y$, so $Y\subsetneq T_1$. This implies that $|Y|<|T_1|$ and $i\notin T_1$. Thus $T_1$ is an element of $\mathcal F$ such that $i\notin T_1$, $E\subsetneq T_1\subsetneq F$, where $E\in\mathcal M$ and $F\in\mathcal N$, and $|Y|<|T_1|$. This contradicts the maximality of $|Y|$, and finishes the proof.
\end{proof}
\begin{lemma}\label{mainprelim}
Suppose that $|\mathcal F|<n$. Then there are at least $\frac{n + 1 - |\mathcal F|}{2}$ elements $i\in [n]$ such that $i\in M$ for some $M\in \mathcal M$, or there are at least $\frac{n + 1 - |\cF|}{2}$ elements $i\in [n]$ such that $i\notin N$ for some $N\in \mathcal N$.
\end{lemma}
\begin{proof}
Let $C=\{i\in [n] : \text{there exists some }W\text{ such that }i\notin W\text{ and }W,W\cup\{i\}\in \mathcal F\}$, $A =\cup_{M\in \mathcal M}M$, and $B=\cup_{N\in\mathcal N} ([n]\setminus N)$. By Lemma~\ref{secondprelim} we have that $A\cup B\cup C=[n]$.

For every $i\in C$, we pick a representative $W_i$ such that $i\notin W_i$ and $W_i, W_i\cup\{i\}\in \mathcal F$. Let $G$ be the graph with vertex set $\mathcal F$ and edge set $E(G)$ such that $\{X,Y\}\in E(G)$ if and only $\{X,Y\}=\{W_i, W_i\cup\{i\}\}$ for some $i\in C$. By construction $G$ has exactly $|C|$ edges as $\{W_i,W_i\cup\{i\}\} \neq \{W_j,W_j\cup\{j\}\}$ for all $i \neq j$. We claim that $G$ is acyclic. Indeed, suppose that $G$ contains a cycle $A_1,A_2,\dots,A_k$ for some $k\geq 3$. Then there exist distinct singletons $i_1,\dots,i_k$ such that $A_{j+1} = A_j \pm \{i_j\}$ for $j\in [k-1]$, and $A_1 = A_k \pm \{i_k\}$. This means that for all $j\neq k$, $i_k\in A_j$ if and only if $i_k\in A_{j+1}$. Hence, by transitivity, $i_k\in A_1$ if and only if $i_k\in A_k$. But this is a contradiction as $A_1 = A_k \pm \{i_k\}$. Therefore, $G$ is an acyclic graph, which means that $|E(G)|\leq |V(G)|-1$. Thus $|M|\leq |\mathcal F|-1$.

We now have that $|A|+|B|\geq |A\cup B|\geq |[n]\setminus C|\geq n+1-|\mathcal F|$. Hence, at least one of $|A|$ or $|B|$ is at least $\frac{n+1-|\cF|}{2}$, as claimed.
\end{proof}
\section{Further analysis of saturated families}
By Lemma~\ref{mainprelim} and the symmetry of the statement under taking complements, we may assume without loss of generality that there are at least $\frac{n + 1 - |\mathcal F|}{2}$ elements $i\in [n]$ such that $i\not\in N$ for some $N\in \mathcal N$.

Let $\mathcal L_0$ be the set of elements above a copy of $\mathcal A_k*\mathcal P_2$ in $\mathcal F$. More precisely, $\mathcal L_0=\{X\in\mathcal P([n]):\text{there exists a subset }\mathcal S\text{ of }\mathcal F\text{ such that }\mathcal S\cup \{X\}\text{ is a copy of }\bullet*\mathcal A_k*\mathcal P_2, \text{ with } X \text{ being the maximal}\\ \text{element of such a copy}\}$. Let $\mathcal L_1$ be the set of minimal elements of $\mathcal L_0$. Finally, let $\mathcal L = \{X \in \mathcal L_1 :\text{there exists no } N\in\mathcal N\text{ such that }N\subseteq X\}$.

Additionally, let $G(\mathcal L)$ be all the sets in $\mathcal F$ that generate $\mathcal L$. More precisely, for every $L\in\mathcal L$, let $G(L)=\{P\in\mathcal F:\text{there exists a subset $\mathcal S$ of }\mathcal F\text{ such that } \mathcal S\cup \{L,P\} \text{ is a copy of $\bullet*\mathcal A_k*\mathcal P_2$ with } L\\\text{as the maximal element of the copy}\}$. Then $G(\mathcal L)=\cup_{L\in\mathcal L}G(L)$. Finally, let $W=\{i\in[n]:i\notin X\text{ for all }X\in\mathcal L\}$. 
\begin{lemma}\label{lemma_no_C_2_in_L_cup_N}
For $\mathcal L$ and $\mathcal N$ defined as above, $\mathcal L$ and $\mathcal N$ are disjoint, and $\mathcal L\cup \mathcal N$ is $C_2$-free.
\end{lemma}
\begin{proof}
We first show that if $L\in \mathcal L$ and $N\in \mathcal N$, then $L\not \subseteq N$. This immediately implies that $\mathcal L$ and $\mathcal N$ are disjoint, since if $Y\in \mathcal L\cap \mathcal N$, then $Y\subseteq Y$, a contradiction. 

Suppose for a contradiction that $L\subseteq N$ for some $L\in \mathcal L$ and $N\in \mathcal N$. By definition, there is a copy of $\mathcal A_k*\mathcal P_2$ in $\mathcal F$ strictly below $L$. Additionally, by the definition of $\mathcal N$, there exists a copy of $\mathcal P_1$ in $\mathcal F$ which has $N$ as its minimal element. This means that every element of this copy of $\mathcal P_1$ is strictly greater than every element of the copy of $\mathcal A_k*\mathcal P_2$ strictly below $L$. Thus, combining these two copies, we  obtain a copy of $\mathcal P_1* \mathcal A_k* \mathcal P_2$ in $\mathcal F$, as illustrated below, a contradiction.
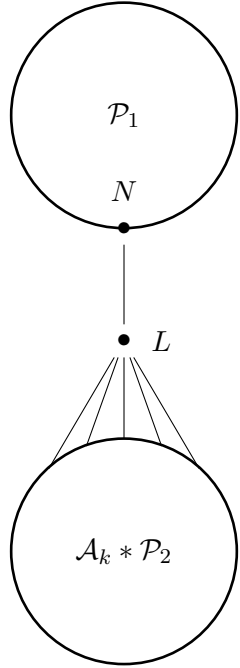
\begin{figure}[h]
\hspace{0.1cm}\\
\centering
\begin{tikzpicture}[node distance=1.5cm and 2cm, baseline=(mid.base)]
\node[label=right:$L$]  (u3) at (4,-0.2) {$\bullet$};
\node[ label= above:$N$] (Y) at (4,1.3) {$\bullet$};
\draw (Y) -- (u3);
\draw[line width = 1pt] (4,2.8) circle(1.5cm);
\draw[line width = 1pt] (4,-3) circle(1.5cm);
\draw (u3)-- (4.5,-1.6);
\draw (u3)-- (3.5,-1.6);
\draw (u3)-- (5,-1.9);
\draw (u3)-- (3,-1.9);
\draw (u3)-- (4,-1.5);
\scoped[every node/.style=]\node at (4, -3) {$\mathcal A_k*\mathcal P_2$};
\scoped[every node/.style=]\node at (4, 2.8) {$\mathcal P_1$};
\end{tikzpicture}
\end{figure}
\FloatBarrier
Thus, $\mathcal L$ and $\mathcal N$ are indeed disjoint, and no element of $\mathcal L$ is contained in any element of $\mathcal N$. Moreover, since both $\mathcal L$ and $\mathcal N$ are antichains, and, by definition, no element of $\mathcal N$ is contained in any element of $\mathcal L$, we obtain that $\mathcal L\cup \mathcal N$ does not contain a copy of $C_2$. 
\end{proof}
\begin{lemma}\label{lemma_almost_C_2_saturated}
Let $S\in \mathcal P([n])\setminus(\mathcal L\cup\mathcal N)$ such that $|S|\geq |\mathcal F|$. Then $\mathcal L\cup \mathcal N\cup \{S\}$ contains a copy of $C_2$.
\end{lemma}
\begin{proof}We first prove the following claim as an intermediary step.
\begin{claim8}
Let $X\in \mathcal P([n])\setminus(\mathcal F\cup \mathcal L\cup \mathcal N)$. Then, either $A\subsetneq X$ for some $A\in\mathcal L\cup\mathcal N$, or $X\subsetneq N$ for some $N\in\mathcal N$.
\end{claim8}
\begin{proof}
Since $X\notin\mathcal F$, $\mathcal F\cup\{X\}$ contains a copy of $\widehat{\mathcal P}$, which must use $X$ as an element. 

If $X$ is in the $\mathcal P_1$ part of this copy, then $X$ is strictly above a copy of $ \mathcal A_k*\mathcal P_2$ in $\mathcal F$. By definition, this means that $X\in \mathcal L_0$, and so there exists $L\in\mathcal L_1$ such that $L\subseteq X$. If there exists $Y\in\mathcal N$ such that $Y\subseteq L$, then $Y\subsetneq X$ as $X\neq Y$ since only one is in $\mathcal N$. On the other hand, if $Y\not\subseteq L$ for all $Y\in \mathcal N$ then $L\in \mathcal L$ and $X\neq L$ as $X\notin\mathcal L$. Thus, in this case, $A\subsetneq X $ for some $A\in \mathcal L\cup\mathcal N$.

Suppose now that $X$ is in the $\mathcal A_k*\mathcal P_2$ part of the copy of $\widehat{\mathcal P}$. Then $X\subsetneq Y$, where $Y$ is the minimal element of the $\mathcal P_1$. By definition, $Y\subseteq N$ for some $N\in \mathcal N$, so $X\subsetneq N$, which finishes the claim.
\end{proof}

Suppose now that $S_0\in\mathcal P([n])\setminus(\mathcal L\cup\mathcal N)$ is such that $\mathcal L\cup\mathcal N\cup\{S_0\}$ is $C_2$-free. By the previous claim, we must have $S_0\in\mathcal F$. We will construct, for all $i\in S_0$, an element $Z_i\in \mathcal F$ such that $S_0\setminus Z_i=\{i\}$. Therefore $S_0$ and $Z_i$ for all $i\in S_0$ are $|S_0|+1$ pairwise distinct elements of $\mathcal F$, which implies that $|\mathcal F|\geq |S_0|+1$. Hence, if $S\in\mathcal P([n])\setminus(\mathcal L\cup\mathcal N)$ such that $|S|\geq|\mathcal F|$, then $\mathcal L\cup\mathcal N\cup\{S\}$ must contain a copy of $C_2$. 

Let now $i\in S_0$. Then $A\not \subseteq S_0\setminus \{i\}$ for all $A\in \mathcal L\cup\mathcal N$, otherwise $A$ and $S_0$ will form a $C_2$ in $\mathcal L\cup\mathcal N\cup\{S_0\}$. If $S_0\setminus\{i\}\in \mathcal F$ then we may set $Z_i=S\setminus\{i\}$, hence assume that $S_0\setminus\{i\}\notin\mathcal F$. By the previous claim, there exists $Y\in\mathcal N$ such that $S_0\setminus \{i\}\subsetneq Y$. Since $S_0\notin\mathcal N$, $S_0\neq N$, which means that $S_0$ and $Y$ must be incomparable. Therefore $i\notin Y$ and $S_0\setminus Y=\{i\}$, which finishes the claim by setting $Z_i=Y$.
\end{proof}

\begin{lemma}\label{lemma_max_min_size_L_N}
For $\mathcal L$ and $\mathcal N$ defined as above, we have that $\max \{|L| : L\in \mathcal L\}\leq |\mathcal F|\) and $\min \{|N|:N\in \mathcal N\}\geq n-|\mathcal F|$.
\end{lemma}
\begin{proof}
We show first that $\min \{|N|:N\in \mathcal N\}\geq n-|\mathcal F|$. Let $N\in \mathcal N$. Then we claim the following.
\begin{claim9}
For every element $i\notin N$, there exists an element $Z_i\in\mathcal F$ such that $Z_i\setminus N=\{i\}$.
\end{claim9}
\begin{proof} If $N\cup\{i\}\in \mathcal F$, then we are done by setting $Z_i = N\cup \{i\}$. Therefore we may assume that $N\cup \{i\}\not\in\mathcal F$, which implies that $\mathcal F\cup\{N\cup\{i\}\}$ contains a copy of $\widehat{\mathcal P}$, which must have $N\cup\{i\}$ as an element.

Suppose first that $N\cup \{i\}$ is not in the $\mathcal P_1$ part of this copy, and let $X$ be the minimal element of its $\mathcal P_1$ part. We must then have $N\subset N\cup\{i\}\subsetneq X$. Moreover, since there is a copy of $\mathcal P_1$ above $X$ in $\mathcal F$, $X\in\mathcal N_0$. However, this is a contradiction as $N$ is a maximal set of $\mathcal N_0$. 

Therefore, $N\cup\{i\}$ must be in the $\mathcal P_1$ part of the copy. Let $Y$ be the union of all the elements in the $\mathcal A_k*\mathcal P_2$ part of the copy. We clearly have that $Y\subseteq N\cup \{i\}$, and $Y$ strictly contains every element of the $\mathcal A_k*\mathcal P_2$ part, since it does not have a unique maximal element ($k\geq2$). If $Y\subseteq N$, then the copy of $\mathcal P_1$ in $\mathcal F$ above $N$ (which exists as $N\in\mathcal N$) combined with the $\mathcal A_k*\mathcal P_2$ part from our copy of $\widehat{\mathcal P}$ form an induced copy of $\widehat{\mathcal P}$ in $\mathcal F$, a contradiction. Hence $Y\subseteq N\cup \{i\}$ and $Y\not \subseteq N$. This means that $Y\setminus N = \{i\}$. By definition, $Y$ is a union of elements of $\mathcal F$, so there exists some $Z_i\in\mathcal F$ such that $Z_i\subseteq Y$ and $i\in Z_i$. Hence $Z_i\setminus N = \{i\}$, which completes the proof of the claim.
\end{proof}
The above claim gives $n - |N|$ distinct elements of $\mathcal F$, which tells us that $|N|\geq n-|\mathcal F|$. Since this holds for every $N\in\mathcal N$, we indeed have that $\min \{|N|:N\in \mathcal N\}\geq n - |\mathcal F|$.

We now show that $\max\{|L|:L\in\mathcal L\}\leq |\mathcal F|$. Let $L\in \mathcal L$. We prove the following claim, which gives us $|L|$ distinct elements of $\mathcal F$. Thus, $|L|\leq |\mathcal F|$, which consequently gives $\max\{|L|:L\in\mathcal L\}\leq|\mathcal F|$.
\begin{claim9}
For every $i\in L$, there exists an element $Z_i\in \mathcal F$ such that $L\setminus Z_i = \{i\}$.
\end{claim9}
\begin{proof}
If $L\setminus\{i\}\in \mathcal F$, then we are done by setting $Z_i= L\setminus\{i\}$, so we may assume that $L\setminus\{i\}\notin \mathcal F$. This means that $\mathcal F\cup \{L\setminus\{i\}\}$ contains a copy of $\widehat{\mathcal P}$ that uses $L\setminus\{i\}$ as one of its elements.

Since $L\in \mathcal L$, hence a minimal element of $\mathcal L_0$, and $L\setminus\{i\}\subsetneq L$, $L\backslash \{i\}$ is not above a copy of $\mathcal A_k*\mathcal P_2$ in $\mathcal F$. Thus $L\setminus\{i\}$ is not in the $\mathcal P_1$ part of our copy of $\widehat{\mathcal P}$. This implies that there exists a copy of $\mathcal P_1$ above $L\setminus\{i\}$ in $\mathcal F$. By looking at the minimal element of such a copy of $\mathcal P_1$, we see that $L\setminus\{i\}\subseteq N$ for some $N\in \mathcal N$. If $L\subseteq N$, there exists a copy of $\mathcal P_1$ above $L$ in $\mathcal F$. On the other hand, since $L\in\mathcal L$, there exists a copy of $\mathcal A_k*\mathcal P_2$ in $\mathcal F$. Combining these two copies, below and above $L$, we obtain a copy of $\widehat{\mathcal P}$ in $\mathcal F$, a contradiction. Therefore, $L\not\subseteq N$ and $L\setminus\{i\}\subseteq N$, implying that $L\setminus N = \{i\}$, hence completing the claim by setting $Z_i=N$.
\end{proof}
\end{proof}
\begin{lemma} Suppose that $|\mathcal F|<\frac{n}{2}$. Then there are at least $\frac{n+1-3|\mathcal F|}{2}$ elements $i\in [n]$ such that $i\in L$ for some $L\in \mathcal L$.
\label{boundingW}
\end{lemma}
\begin{proof}
Recall that there are at least $\frac{n + 1-|\mathcal F|}{2}$ elements $i\in [n]$ such that $i\notin N$ for some $N\in \mathcal N$. Let $\mathcal B=\{i\in [n]:\text{ there exists } N\in\mathcal N\text{ such that }i\notin N$. By assumption, we have that $|\mathcal B|\geq \frac{n+1-|\mathcal F|}{2}$. For all $i\in \mathcal B$ let $N_i\in\mathcal N$ be an element of minimal cardinality  that does not contain $i$. In other words, $N_i$ is an element of $\{X\in \mathcal N: i\notin X\}$ with minimal cardinality. Let $\mathcal C = \{i \in\mathcal B: N_i = N_j \text{ for some $j\in\mathcal B\setminus\{i\}$}\}$. By construction, we have that $|\mathcal C|\geq |\mathcal B| - |\mathcal N|\geq \frac{n+1-3|\mathcal F|}{2}$. We now show that if $i\in\mathcal C$, then $i\in L$ for some $L\in\mathcal L$, which finishes the proof.

Suppose that there exists $i\in \mathcal C$ such that $i\notin L$ for all $L\in\mathcal L$. By Lemma~\ref{lemma_max_min_size_L_N}, we have that for all $j\in\mathcal B$, $|N_j|\geq n-|\mathcal F|\geq |\mathcal F|$, since $|\mathcal F|\leq\frac{n}{2}$. In particular, this means that for all $j\in N_i$ we have that $|(N_i\setminus\{j\})\cup\{i\}|=|N_i|\geq|\mathcal F|$. By Lemma~\ref{lemma_almost_C_2_saturated}, this means that for all $j\in N_i$, there exists $X_j\in \mathcal L\cup\mathcal N$ such that $X_j$ is comparable (or equal) to $(N_i\setminus\{j\})\cup \{i\}$.

If $X_j \subsetneq (N_i\setminus\{j\})\cup \{i\}\), then $|X_j|<|N_i|$. By assumption, $N_i = N_{i_2}$ for some $i_2\in\mathcal B \setminus \{i\}$. Therefore, $i_2\notin X_j$ and $|X_j|<||N_i|=N_{i_2}|$, which, by the minimality of $|N_{i_2}|$ implies that $X_j\notin \mathcal N$, so $X_j\in\mathcal L$. By our assumption, we the have $i\notin X_j$. However, this now means that $X_j\subset N_i$, which contradicts Lemma~\ref{lemma_no_C_2_in_L_cup_N} as $X_j$ and $N_i$ would induce a $C_2$ in $\mathcal L\cup\mathcal N$. Therefore, $(N_i\setminus\{j\})\cup\{i\}\subseteq X_j$ for all $j\in N_i$. We observe that $X_j\neq N_i$, which means that $X_j$ and $N_i$ are incomparable by Lemma~\ref{lemma_no_C_2_in_L_cup_N}. Since $N_i\setminus\{j\}\subset X_j$, we must have that $j\notin X_j$. Thus, if $j, k \in N_i$ are distinct, then $X_j\neq X_k$. Moreover, by Lemma~\ref{lemma_max_min_size_L_N}, we have that $|\mathcal F|<n-|\mathcal F| \leq|N_i|\leq|X_j|$, which implies that $X_j\in \mathcal N$. However, this would mean that $\{X_j : j\in N_i\}$ is a set of $|N_i| >\frac{n}{2}$ elements of $\mathcal N\subseteq\mathcal F$, a contradiction.
\end{proof}
\section{Proof of the main result}
We are now ready to prove that the saturation number of $\mathcal P_1*\mathcal A_k*\mathcal P_2$ is at least linear. Having the above preliminary lemmas, we are going to use a similar upper bounding technique as the one used in \cite{diamondlinear}. Although the setup needs a straightforward adjustment, from the diamond to the general $\mathcal P_1*\mathcal A_k*\mathcal P_2$, which we present below, the proof of Lemma~\ref{boundingsizecalAcalB} (\cite[Lemma 6]{diamondlinear}) is identical, and hence we do not include it here.

Let $X$ be a set and $m$ a positive integer such that $2m+1\leq |X|$. We define $\mathcal L(X,m)$ to be the set of disjoint pairs of subsets of $\mathcal P(X)$, $(\mathcal G,\mathcal H)$ such that all elements of $\mathcal G$ have size at most $m$, all elements of $\mathcal H$ have size at least $n-m$, and for every $i\in X$, there exists an element $G\in \mathcal G$ such that $i\in G$. Additionally, any induced copy of $C_2$ in $\mathcal G\cup \mathcal H$ must be contained in $\mathcal H$, and for any $A\in \mathcal P(X)$ with $m\leq|A|\leq|X|-m$, there exists an element $B\in\mathcal G\cup\mathcal H$ such that either $B\subseteq A$, or $A\subseteq B$.

Next, let $\mathcal I,\mathcal J\subseteq \mathcal P(X)$ and define $\mathcal T_0(\mathcal I,\mathcal J)=\{A\in \mathcal P(X):B\not\subseteq A\text{  for all }B\in\mathcal J, \text{ and }\exists \mathcal S\subseteq\mathcal I \text{ such that } S\cup\{A\} \text{ is a copy of $\bullet*\mathcal A_k*\mathcal P_2$, with } A \text { being its maximal element}\}$.

We define $\mathcal T(\mathcal I,\mathcal J)$ to be the set of minimal elements of $\mathcal T_0(\mathcal I, \mathcal J)$.

Finally, define $\mathcal L_*(X,m)$ to be the set of pairs of $\mathcal P(X)$, $(\mathcal I,\mathcal J)$, such that $(\mathcal T(\mathcal I, \mathcal J), \mathcal J)\in\mathcal L(X,m)$, and all sets of $\mathcal I$ have size at most $m$. Lastly, define $f(n,m) = \min \{|\mathcal I\cup\mathcal J|:(\mathcal I,\mathcal J)\in\mathcal L_*([n],m)\}$.   

We then have the following lemma, the proof of which is the same as in \cite[Lemma 6]{diamondlinear}.
\begin{lemma}\label{boundingsizecalAcalB} For any positive integers $n$ and $m$ such that $n\geq 2m+1$, we have $f(n,m)\geq n - 2m$.
\end{lemma}
We will lower bound $G(\mathcal L)\cup\mathcal N$ with the help of Lemma~\ref{boundingsizecalAcalB}, after removing the set $W$, which we know how to upper bound by Lemma~\ref{boundingW}.
\begin{theorem}
Let $n\geq 1$, and $\mathcal F$ a $\mathcal P_1*\mathcal A_k*\mathcal P_2$-saturated family with ground set $[n]$. Then $|\mathcal F|\geq\frac{n+1}{9}$.
\label{maintheorem}
\end{theorem}
\begin{proof}
Since $\frac{n+1}{9}<1$ for $n<8$, the result is trivially true for $n<8$, so we may assume that $n\geq 8$. If $|\mathcal F|>\frac{n-1}{7}$, then we are done, so we may assume that $|\mathcal F|\leq\frac{n-1}{7}$.

Recall that $W=\{i\in[n]:i\notin X\text{ for all } X\in\mathcal L\}$. By the Lemma~\ref{boundingW}, we have that $|W|\leq \frac{n-1+3|\mathcal F|}{2}$. Since $\mathcal L$ is a subset of the set of minimal elements that are above a copy of $\mathcal A_k*\mathcal P_2$ in $\mathcal F$, if $L\in\mathcal L$, then $L=X_1\cup\dots X_k$ for some $X_1,\dots,X_k\in G(\mathcal A)$. That means that if $i\in W$, then $i\notin X$ for all $X\in G(\mathcal L)$. Conversely, if $i\notin X$ for all $X\in G(\mathcal L)$, then by the same minimality argument we have $i\in W$. Therefore $W=\{i\in[n]:i\notin X\text{ for all }X\in G(\mathcal L)\}$. Let $\widehat{N}=\{N\setminus W:N\in\mathcal N\}$.

Moreover, by Lemma~\ref{lemma_max_min_size_L_N}, we have that every set in $\mathcal L$ has size at most $|\mathcal F|$, and so every set in $G(\mathcal L)$ has size at most $|\mathcal F|$. Also, since every set in $\mathcal N$ has size at least $n-|\mathcal F|$, and $|\mathcal F|\leq\frac{n-1}{7}$, we must have that $G(\mathcal L)$ and $\mathcal N$ are disjoint. Furthermore, if $\widehat{N}\in\widehat{\mathcal N}$, then $|\widehat{N}|\geq n-|\mathcal F|-|W|=|[n]\setminus W|-|\mathcal F|$. This means that every set in $\widehat{N}$ has size at least $\frac{n+1-5|\mathcal F|}{2}$, which is greater that $|\mathcal F|$, as $|\mathcal F|\leq\frac{n-1}{7}$. Thus $\mathcal L$ and $\widehat{\mathcal N}$ are disjoint, as every element of $\mathcal L$ has size at most $|\mathcal F|$, and every element of $\widehat{\mathcal N}$ has size at least $|[n]\setminus W|-|\mathcal F|$.
\begin{claim5}$(G(\mathcal L),\widehat{\mathcal N})\in\mathcal L_*([n]\setminus W, |\mathcal F|)$.   
\end{claim5}
\begin{proof}
First, since $\mathcal L_*(X,m)$ is defined only when $2m+1\leq |X|$, we need to have $2|\mathcal F|+1\leq n-|W|$, or equivalently, $2|\mathcal F|+|W|+1\leq n$. This is true since $|W|\leq\frac{n-1+3|\mathcal F|}{2}$ and $|\mathcal F|\leq\frac{n-1}{9}$. Also, by construction, $G(\mathcal L)$ and $\widehat{\mathcal N}$ have ground set $[n]\setminus W$.

Since we already established that all sets of $G(\mathcal L)$ have size at most $|\mathcal F|$, we are left to prove that $(\mathcal T(G(\mathcal L),\widehat{\mathcal N}),\widehat{\mathcal N})\in\mathcal L([n]\setminus W)$. However, since $G(\mathcal L)$ generates $\mathcal L$ and, by cardinality, no element of $\mathcal N$ or $\widehat{\mathcal N}$ can be a subset of an element of $\mathcal L$, we notice that $\mathcal T(G(\mathcal L),\widehat{\mathcal N})=\mathcal L$.

We therefore must show that $(\mathcal L,\widehat{\mathcal N})\in\mathcal L([n]\setminus W, |\mathcal F|)$. We have already showed that $\mathcal L$ and $\widehat{\mathcal N}$ are disjoint, that every set of $\mathcal L$ has size at most $|\mathcal F|$ and every set of $\widehat{\mathcal N}$ has size at least $|n\setminus W|-|\mathcal F|$, and that $\mathcal L$ is a cover of $[n]\setminus W$ by construction. Suppose now that $\mathcal L\cup\widehat{\mathcal N}$ contains a copy of $C_2$, say $X_1\subsetneq X_2$, such that at least one of them is in $\mathcal L$. Since $\mathcal L$ is an antichain, we must have one element in $\mathcal L$ and one in $\widehat{\mathcal N}$, and since all sets of $\mathcal A$ have sizes less than the size of any element in $\widehat{\mathcal N}$, we must have $X_1\in\mathcal A$ and $X_2\in\widehat{\mathcal N}$, so $X_2=N\setminus W$ for some $N\in\mathcal N$. But now $X_1\subsetneq X_2=N\setminus W\subseteq N$, which contradicts the fact that $\mathcal L\cup\mathcal N$ is $C_2$-free (Lemma~\ref{lemma_no_C_2_in_L_cup_N}). Finally, suppose $X\in\mathcal P([n]\setminus W)$ is a set such that $|\mathcal F|\leq |X|\leq n-|W|-|\mathcal F|$. If $X\in\mathcal L$, we have $X\subseteq X$. If $X\in\mathcal N$, we have $X\setminus W\subseteq X$ and $X\setminus W\in\widehat{\mathcal N}$. Thus, we may assume that $X\notin \mathcal L\cup\mathcal N$, which together with the fact that $|X|\geq|\mathcal F|$ implies, by Lemma~\ref{lemma_almost_C_2_saturated}, that there exists $Y\in\mathcal L\cup\mathcal N$ such that $X\subset Y$, or $Y\subset X$. If $X\subset Y$, by cardinality we must have $Y\in\mathcal N$, and so $X\subset Y\setminus W$, and $Y\setminus W\in\widehat{\mathcal N}$. If on the other hand $Y\subset X$, we similarly must have $Y\in\mathcal L$, which finishes the proof of the claim.

\end{proof}
We now see that, since $G(\mathcal L)$ and $\mathcal N$ are disjoint and subsets of $\mathcal F$, $|\mathcal F|\geq |G(\mathcal L)|+|{\mathcal N}|\geq |G({\mathcal L})|+|\widehat{\mathcal N}|\geq f(n-|W|,|\mathcal F|)\geq n-|W|-2|\mathcal F|$, where the last inequality comes from Lemma~\ref{boundingsizecalAcalB}. Together with the fact that $|W|\leq\frac{n-1+3|\mathcal F|}{2}$, we get that $|\mathcal F|\geq\frac{n+1-7|\mathcal F|}{2}$, or equivalently $|\mathcal F|\geq\frac{n+1}{9}$, as desired.
\end{proof}
\section{Saturation for multipartite posets}
In this section we construct saturated families of linear size for all multipartite posets. This, together with Theorem~\ref{maintheorem} shows that the saturation number for all multipartite posets is linear, with the exception of a chain, for which the saturation number is already known to be bounded \cite{gerbner2013saturating}.

Let $K_{n_1,\dots,n_k}$ denote the complete poset with $k$ layers of sizes $n_1,\dots,n_k$, in this order, where $n_1$ is the size of the bottom layer. In other works, we have that $K_{n_1,\dots,n_k}=\mathcal A_{n_k}*\dots*\mathcal A_{n_1}$.
\begin{theorem}
Let $n_1,\dots,n_k$ be positive integers. Then $\text{sat}^*(n,K_{n_1,\dots,n_k})=O(n)$.
\end{theorem}
\begin{proof}
If there exists $i\in[k]$ such that $n_i=1$, but $n_{i-1}$ and $n_{i+1}$ are not 1, then by Corollary~\ref{poset-point-poset-corollary} we have that $\text{sat}^*(n, K_{n_1,\dots,n_k})\leq\text{sat}^*(n, K_{n_1,\dots,n_{i-1},n_{i+1},\dots,n_k})$. Since we want to show that the saturation number is at most linear, we may assume that $K_{n_1,\dots,n_k}$ is equal to $\mathcal P_l*\mathcal P_{l-1}*\dots*\mathcal P_2*\mathcal P_1$ for some positive integer $l$, where each $\mathcal P_i$ is either an antichain of size at least two, or a chain of size at least two. Moreover, for all $i\in[l-1]$, $\mathcal P_i$ and $\mathcal P_{i+1}$ are not both chains (otherwise we combine them).

Let $n\geq |\mathcal P_1|+\dots|\mathcal P_l|$. We will construct a $\mathcal P_l*\dots*\mathcal P_1$-saturated family in $\mathcal P([n])$ as follows. We will start with two constant size families, one sitting inside a fixed hypercube that is low in $\mathcal P([n])$, while the other being (morally) the `reflection' of the first family. We then arbitrarily saturate the union of these families. To do so, we need to introduce some helpful notation.

Let $T$ be a set of non-negative integers, and $s\in \mathbb N$. We denote by $\binom{[s]}{T}$ the collection of all subsets of $[s]$ that have sizes equal to some element of $T$. More precisely, $\binom{[s]}{T}=\{X\in \mathcal P([s]):|X|\in T\}$.

Let $m=\sum_{i=1}^{l} |\mathcal P_i|-1$ and, for all $1\leq i\leq l-1$ define $T_i$ as follows. If $\mathcal P_i$ is an antichain, then $T_i$ is the singleton $\{\sum_{j=1}^{i-1}|\mathcal P_j| + 1\}$. If $\mathcal P_i$ is a chain, then $T_i$ is the interval $[\sum_{j=1}^{i-1}|\mathcal P_j| + 1,\sum_{j=1}^{i}|\mathcal P_j|]$.

Similarly, for each $2\leq i\leq l$, we define $U_i$ as follows. If $\mathcal P_i$ is an antichain, then $U_i$ is the singleton $\{\sum_{j=i+1}^{l}|\mathcal P_j| + 1\}$. If $\mathcal P_i$ is a chain, then $U_i$ is the interval $[\sum_{j=i+1}^l |\mathcal P_j| + 1,\sum_{j=i}^{l}|\cP_j|]$.

We first start with the following `bottom' and `top' families, both of constant size for a fixed poset of the considered form. Let $\mathcal F_b = \bigcup_{i=1}^{l-1} \binom{[m]}{T_i} \) and $\mathcal F_t=\bigcup_{i=2}^l \left\{ [n]\backslash X : X\in \binom{[m]}{U_i}\right\}$.
\begin{lemma1} The family $\mathcal F_b\cup\mathcal F_t$ is $\mathcal P_l*\dots*\mathcal P_1$-free.
\end{lemma1}
\begin{proof}
Suppose that $\mathcal F_b\cup\mathcal F_t$ contain such a copy, which we call $\mathcal P_l*\dots*\mathcal P_1$. We remark that if $i<j$, then every element of $\mathcal P_i$ is smaller than every element of $\mathcal P_j$. Additionally, since $n>2m$, every set in $\mathcal F_b$ has size strictly less than the size of any set in $\mathcal F_t$. It is easy to see that the length of the longest chain in $\mathcal F_b$ is $\sum_{i=1}^{l-1}|T_i|$, while the length of the longest chain in $\mathcal P_l*\dots*\mathcal P_1$ is the number of $\mathcal P_i$'s that are antichains plus the sum of the sizes of the ones that are chains, i.e. $\sum_{i=1}^l|T_i|$. This tells us that $\mathcal F_b$ does not contain a copy of $\mathcal P_l*\dots*\mathcal P_1$. Similarly, $\mathcal F_t$ does not contain a copy of $\mathcal P_l*\dots*\mathcal P_1$. Therefore, by looking at the smallest $i\in[l]$ for which $\mathcal P_i\not\subseteq\mathcal F_b$, we get that $\mathcal P_1*\dots*\mathcal P_{i-1}\subseteq \mathcal F_b$ and $\mathcal P_{i+1}*\dots*\mathcal P_l\subseteq \cF_t$.

Let $A = \bigcup_{X\in\mathcal P_1*\dots*\mathcal P_{i-1}} X $ and $B = \bigcap_{X\in \mathcal P_{i+1}*\dots*\mathcal P_l} X$.

\begin{claim7}
\( |A|\geq \sum_{j=1}^{i-1} |\cP_j| \).
\end{claim7}
\begin{proof}
We will prove something more general. We will show that for every $1\leq r\leq l$, if $C$ is a set such that there exists an induced copy of $\mathcal P_{r}*\dots*\mathcal P_1$ in $\mathcal F_b$ below it, i.e $\mathcal P(C)\cap \mathcal F_b$ contains an induced copy of $\mathcal P_{r}*\dots*\mathcal P_1$, then $|C|\geq \sum_{j=1}^r |\mathcal P_j|$. 

We prove this is by induction on $r$. When $r=1$ it is a trivial check, both when $\mathcal P_1$ is an antichain or a chain. Assume now that $r\geq 2$ and that the result holds for $1,\dots,r-1$. Moreover, since $\mathcal F_b$ does not contain a copy of $\mathcal P_l*\dots*\mathcal P_1$, we must have $r\leq l-1$.

Consider a copy of $\mathcal P_{r}*\dots*\mathcal P_1$ in $\mathcal F_b$, which we call $\mathcal P_{r}*\dots*\mathcal P_1$ for simplicity. Let $C$ be a set above it. If any element in $\mathcal P_r$ has size at least $\sum_{j=1}^r |\mathcal P_j|$, then we are done. We may therefore assume that all sets in $\mathcal P_r$ have size less than $\sum_{i=1}^r|\mathcal P_i|$.
        
Let $D$ be the union of the elements of $\mathcal P_{r-1}*\dots*\mathcal P_1$. By the inductive hypothesis, we have that $|D|\geq \sum_{j=1}^{r-1} |\mathcal P_j|$. If $\mathcal P_{r-1}$ is a chain, then $D$ is the top point of $\mathcal P_{r-1}$, and so it is a strict subset of all elements in $\mathcal P_r$. Hence, all sets of $\mathcal P_r$ have sizes at least $\sum_{j=1}^{r-1}|\mathcal P_j|+1$. If $\mathcal P_{r-1}$ is an antichain, then, by construction, no set in $\mathcal F_b$ has size $\sum_{j=1}^{r-1}|\mathcal P_j|$. This means that either $|D|=\sum_{j=1}^{r-1}|\mathcal P_j|$ in which case $D\notin\mathcal F_b$ and so $D$ is a strict subset of all elements in $\mathcal P_r$, or $|D|\geq\sum_{j=1}^{r-1}|\mathcal P_j|+1$. In both cases we get that all sets in $\mathcal P_r$ have sizes between $\sum_{j=1}^{r-1} |\mathcal P_j|+1$ and $\sum_{j=1}^r |\mathcal P_j|-1$ inclusive. Since all these sets are in $\mathcal F_b$, we get that $\mathcal P_r$ is contained in $\binom{[m]}{T_r}$. 

If $\mathcal P_r$ is a chain, then its elements have pairwise distinct sizes. Since $|T_r| = |\mathcal P_r|$, this implies that the top element of the chain must have size equal to $\max T_r = \sum_{j=1}^r |\mathcal P_r|$, and so $|C|\geq\sum_{j=1}^r |\mathcal P_r|$.        
        
If $\mathcal P_r$ is an antichain, then all its elements strictly contain $D$. If we remove $D$ from all of them, we get an antichain of the same size. More precisely, let $\widehat{\mathcal P_r}=\{X\setminus D : X\in \mathcal P_r\}$, which is isomorphic to $\mathcal P_r$. Let also $\widehat{C} = \cup_{X\in \widehat{\mathcal P_r}} X$, which is the same as $C\setminus D$. It is therefore enough to show that $|\widehat{C}|\geq |\mathcal P_r|$. 

Since $\mathcal P_r$ is an antichain, we have that $T_r=\{\sum_{j=1}^{r-1}|\mathcal P_j|+1\}$, and so all sets of  $\widehat{\mathcal P_r}$ are of the form $X\setminus D$ for some $X\in \binom{[m]}{\sum_{j=1}^{r-1}|\mathcal P_j|+1}$. Since by construction $D$ is a subset of $[m]$, we get that all sets of $\widehat{\mathcal P_r}$ are in $\binom{[m]\setminus D}{\sum_{j=1}^{r-1} |P_j|+1-|D|}$. Moreover, all sets of $\mathcal P_r$ are subsets of $\widehat{C}$, hence $\widehat{\mathcal P_r}\subseteq  \binom{\widehat{C}}{\sum_{j=1}^{r-1} |\mathcal P_j|+1-|D|}$, which implies that $|\widehat{\mathcal P_r}|\leq \binom{|\widehat{C}|}{\sum_{j=1}^{r-1} |\mathcal P_j|+1 - |D|}$. Since $|D|\geq \sum_{j=1}^{r-1}|\mathcal P_j|$, we indeed get that $|\widehat{\mathcal P_r}|\leq|\widehat{C}|$, which completes the proof of the claim.
\end{proof}
\begin{claim7}$|B|\leq n-\sum_{j=i+1}^{l} |\mathcal P_j|$.
\end{claim7}
\begin{proof}
By taking complements, the result follows by Claim 1.
\end{proof}
By construction, $A$ is a union of elements of $\mathcal F_b$ and $B$ an intersection of elements in $\mathcal F_t$. Thus $A\subseteq [m]$ and $[n]\setminus[m]\subseteq B$. Since $A$ is below $\mathcal P_i$ and $B$ is above $\mathcal P_i$, we also have $A\subseteq B$. We can therefore write $B = ([n]\setminus [m]) \cup A \cup W$ for some $W\subseteq [m]\backslash A$. Since $|B|\leq n-\sum_{j=i+1}^{l}|\mathcal P_j|$, we get that $|A|+|W|=|B\cap[m]|\leq\sum_{j=1}^{i}|\mathcal P_j|-1$.
    
We observe that If $X\in\mathcal P_i$ then $A\subseteq X\subseteq B$. If $X\in\mathcal F_b$, then $|B\cap [m]| \geq |X|\geq |A|\geq\sum_{j=1}^{i-1}|\mathcal P_i|$, and so $\sum_{j=1}^{i} |\cP_j|-1\geq |X|\geq \sum_{j=1}^{i-1}|\cP_j|$. By looking at whether $\mathcal P_{i-1}$ is a chain or an antichain (the same way as in Lemma 1 above), we get that $|X|\in T_i$. Similarly, if $X\in\mathcal F_t$ we get that $|[n]\setminus X|\in U_i$.
 
Suppose first that $\mathcal P_i$ is a chain, and let $X_1,\dots,X_{|\mathcal P_i|}$ be its elements, where $X_1\subsetneq X_2\subsetneq \dots\subsetneq X_{|\cP_i|}$. As discussed above, for every $X_k\in \mathcal P_i$, we have that $\sum_{j=1}^{i} |\mathcal P_j|-1\geq |X_k\cap [m]|\geq \sum_{j=1}^{i-1}|\mathcal P_j|$. Thus, by the pigeonhole principle, there exist $k_1\neq k_2$ such that $|X_{k_1}\cap[m] |=|X_{k_2}\cap [m]|$, and, without loss of generality, we may assume that $k_1<k_2$. This means that $X_{k_1}\subsetneq X_{k_2}$. 

Since $X_{k_1}\cap [m] = X_{k_2}\cap[m]$, $X_{k_1}$ and $X_{k_2}$ cannot be both in $\mathcal F_b$ or both in $\mathcal F_{t}$. We therefore must have $X_{k_1}\in\mathcal F_b$, $X_{k_2}\in\mathcal F_t$, and $X_{k_2}=[n]\setminus[m]\cup X_{k_1}$. This immediately implies that $k_2 = k_1+1$. Now, for all $1\leq j\leq k_1$, we have $X_j\in \cF_b$, which by the above gives us that $ |X_j|\in T_i$. Therefore we have that $\sum_{k=1}^{i-1}|\mathcal P_k|+1\leq |X_1|<|X_2|<\dots<|X_{k_1-1}|<|X_{k_1}|$, which implies that $|X_{k_1}|\geq\sum_{k=1}^{i-1}|\mathcal P_k|+k_1 $. Similarly, by taking complements, we get that $|[n]\setminus X_{k_1+1}|\geq \sum_{k=i+1}^l |\mathcal P_k | + (|\mathcal P_i| - k_1)$, or equivalently, $|[m]\setminus X_{k_1}|\geq \sum_{k=i}^l|\mathcal P_k|-k_1$. Rearranging the last inequality we get that $|X_{k_1}|\leq\sum_{k=1}^{i-1}|\mathcal P_k|+k_1-1$, a contradiction.

Finally, suppose that $\mathcal P_i$ is an antichain. Let its elements be $A_{t_1}, \dots A_{t_{|\mathcal P_i|}}$. If $A_{t_k}\in\mathcal F_b$, then $|A_{t_k}|=\sum_{j=1}^{i-1}|\mathcal P_j|+1$ and $A_{t_k}=A\cup T_{t_j}$ for some $T_{t_j}\subseteq W$. Since $|A|\geq\sum_{j=1}^{i-1}|\mathcal P_j|$, then $T_{i}$ is either empty or a singleton. However, it cannot be empty as it would then be equal to $A$ which is comparable to everything in $\mathcal P_i$. By symmetry, if $A_{t_k}\in\mathcal F_t$, then $[n]\setminus A_{t_k}=([n]\setminus B)\cup L_{t_k}$ where $L_{t_k}$ is a singleton. In this case, since $A_{t_k}=([n]\setminus[m])\cup A\cup T_{t_k}$ for some $T_{t_k}\subseteq W$, we must have $T_{t_k}=W\setminus L_{t_k}$. In conclusion, those elements of $\mathcal P_i$ that are in $\mathcal F_b$ are equal to $A$ union a singleton of $W$, and those that are in $\mathcal F_t$ are equal to $([n]\setminus[m])\cup A$ union the complement of a singleton in $W$.

Let $X_{t_k}=A\cup T_{t_k}$ if $X_{t_k}\in\mathcal F_b$, and $X_{t_k}=([n]\setminus[m])\cup A\cup T_{t_k}$ if $X_{t_k}\in\mathcal F_t$. In light of the above observations, it is trivial to see that $T_{t_1},\dots,T_{t_{|\mathcal P_i|}}$ form an anti chain in $W$. Moreover, this antichain is comprised of singletons and complements of singletons. The LYM inequality then tells us that this antichain has size at most $|W|$, thus $|\mathcal P_i|\leq|W|$.

However, we have shown above that $|A|+|W|\leq\sum_{j=1}^i|\mathcal P_j|-1$, which combined with the fact that $|A|\geq\sum_{j=1}^{i-1}|\mathcal P_j|$, gives that $|W|\leq|\mathcal P_i|-1$,  contradiction.
\end{proof}

We therefore have that $\mathcal F_b\cup\mathcal F_t$ is $\mathcal P_l*\dots*\mathcal P_1$-free. We now saturate this family arbitrarily and obtain a $\mathcal P_1*\dots*\mathcal P_l$-saturated family $\mathcal F$ that contains $\mathcal F_b\cup\mathcal F_t$. We are left to prove that $|\mathcal F| = O(n)$.

Let $T$ be a subset of $[m]$, and define the set $C_T = \{X\in \mathcal F : X\cap [m] = T\}$. We obviously have that $\mathcal F= \cup_{T\subseteq [m]}C_T$. In order to finish the proof, it is therefore enough to show that for every $T\subseteq[m]$, we have that $|C_T|=O(n)$. We start with the following lemma.

\begin{lemma1}\label{lemmabelowC_T}
Let $2\leq i\leq l$ and $T$ a subset of $[m]$ such that $\sum_{j=1}^{i-1}|\mathcal P_j|\leq |T|$. Then $\mathcal P(T)\cap\mathcal F_b$ contains an induced  copy of $\mathcal P_1*\dots*\mathcal P_{i-1}$.   
\end{lemma1}
\begin{proof}
We will prove this by induction on $i$. When $i=2$, it is easy to check that there is a $\mathcal P_1$ below $T$ (inclusively), both when $\mathcal P_1$ is a chain or an antichain. Assume now that $i\geq 3$ and that the lemma is true for $i-1$. Without loss of generality, we may assume that $T = [t]$ for some $t\geq\sum_{j=1}^{i-1}|\mathcal P_j|$. Let $T'=[\sum_{j=1}^{i-2}|\mathcal P_j|]$. By our inductive hypotheses, there exists a copy of $\mathcal P_1*\dots*\mathcal P_{i-2}$ in $\mathcal F_b$ below $T'$. It is therefore enough to show that there exists a copy of $\mathcal P_{i-1}$ in $\{X\in \mathcal F_b : [\sum_{j=1}^{i-2}|\mathcal P_j|]\subsetneq X\subseteq [\sum_{j=1}^{i-1}|\mathcal P_j|]\}$. 

If $\mathcal P_{i-1}$ is a chain, then $\{[q] :\sum_{j=1}^{i-2}|\mathcal P_j|+1\leq  q\leq\sum_{j=1}^{i-1}|\mathcal P_j|\}$ is such a copy, while if $\mathcal P_{i-1}$ is an antichain, then $\{[\sum_{j=1}^{i-2}|\mathcal P_j|]\cup \{a\} : \sum_{j=1}^{i-2}|\mathcal P_j|+1\leq a\leq \sum_{j=1}^{i-1}|\mathcal P_j|\}$ is again such a copy. This finishes the induction hypothesis and consequentially the proof of the lemma.
\end{proof}
By symmetry, we also get the following.
\begin{lemma1}\label{lemmaaboveC_T}
Let $1\leq i\leq l-1$ and $T$ a subset of $[m]$ such that $|T|\leq\sum_{j=1}^{i}|\mathcal P_j|-1$. Then there exists a copy of $\mathcal P_{i+1}*\dots*\mathcal P_{l}$ above $T\cup([n]\setminus[m])$ in $\mathcal F_t$. 
\end{lemma1}

\begin{lemma1}\label{chainC_T}
Let $i\in[l]$ and $T$ a subset of $[m]$ such that $\sum_{j=1}^{i-1}|\mathcal P_j|\leq |T|\leq\sum_{j=1}^{i}|\mathcal P_j|-1$. If $\mathcal P_i$ is a chain, then $C_T\subseteq \mathcal F_b\cup\mathcal F_t$.
\end{lemma1}
\begin{proof}
Let $X\in \mathcal P([n])\setminus(\mathcal F_b\cup\mathcal F_t)$ such that $X\cap [m] = T$. We will show that $\mathcal F_b\cup \mathcal F_t\cup \{X\}$ contains a copy of $\mathcal P_1*\dots*\mathcal P_l$, hence $X\notin \mathcal F$ and consequently $X\notin\ C_T$ -- in other words, $C_T\subseteq\mathcal F_b\cup\mathcal F_t$. Without loss of generality, we may assume that $T = [t]$ for some $t\in [\sum_{j=1}^{i-1}|\mathcal P_j|,\sum_{j=1}^{i}|\mathcal P_j|-1]$. 

If $2\leq l\leq l-1$, we have by Lemma~\ref{lemmabelowC_T} and Lemma ~\ref{lemmaaboveC_T} that there exists a copy of $\mathcal P_1*\dots*\mathcal P_{i-1}$ in $\mathcal F_b$ below $[\sum_{j=1}^{i-1}|\cP_j|]$, and a copy of $\mathcal P_{i+1}*\dots*\mathcal P_l$ in $\mathcal F_t$ above $[\sum_{j=1}^{i}|\mathcal P_j|-1]\cup ([n]\setminus[m])$. Therefore, it suffices to show that there is a copy of $\mathcal P_i$ in $\mathcal F_b\cup \mathcal F_t \cup \{X\}$ strictly between $[\sum_{j=1}^{i-1}|\mathcal P_j|]$ and $[\sum_{j=1}^{i}|\mathcal P_j|-1]\cup ([n]\setminus[m])$. However, $\{[q] : \sum_{j=1}^{i-1}|\mathcal P_j| +1\leq q\leq t\}\cup \{X\}\cup \{[q]\cup([n]\setminus[m]) : t\leq q\leq\sum_{j=1}^{i} |\mathcal P_j|-2 \}$ is such a copy of $\mathcal P_i$.

If $i=1$, by Lemma~\ref{lemmaaboveC_T} there exists a copy of $\mathcal P_{2}*\dots*\mathcal P_l$ in $\mathcal F_t$ above $[|\mathcal P_1|-1]\cup ([n]\setminus[m])$. But then $\{[q]:1\leq q\leq t\}\cup\{X\}\cup\{[q]\cup([n]\setminus[m]):q\leq t\leq |\mathcal P_2|-1\}$ is a copy of $\mathcal P_1$ is $\mathcal F_b\cup\mathcal F_t\cup\{X\}$ strictly below $[|\mathcal P_1|-1]\cup ([n]\setminus[m])$. A completely analogous argument deals with the case $i=l$.
 
\end{proof}\label{antichainC_T}
\begin{lemma1}Let $i\in[l]$ and $T$ a subset of $[m]$ such that $\sum_{j=1}^{i-1}|\mathcal P_j|\leq |T|\leq\sum_{j=1}^{i}|\mathcal P_j|-1$. If $\mathcal P_i$ is an antichain, then $|C_T|\leq|\mathcal P_i|(n+1)$.
\end{lemma1}
\begin{proof}
Assume without loss of generality that $T=[t]$ for some $\sum_{j=1}^{i-1}|\mathcal P_j|\leq t\leq\sum_{j=1}^{i}|\mathcal P_j|-1$. Suppose that there exists and antichain of size $|\mathcal P_i|$ in $C_T$, say $X_1,\dots, X_{|\mathcal P_k|}$. Then $[t]\subsetneq X_i\subsetneq [t]\cup([n]\setminus[m])$ for all $1\leq i\leq |\mathcal P_i|$. If $2\leq i\leq l-1$, then by Lemma~\ref{lemmabelowC_T} we get a copy of $\mathcal P_1*\dots\mathcal P_{i-1}$ below $[t]$, and by Lemma~\ref{lemmaaboveC_T} we get a copy of $\mathcal P_{i+1}*\dots*\mathcal P_l$ above $[t]\cup([n]\setminus [m])$. These two copies together with the antichain in $C_T$ give us a copy of $\mathcal P_1*\dots*\mathcal P_l$ in $\mathcal F$, a contradiction. A similar argument also gives the same contradiction if $i=1$ or $i=l$.

Therefore $C_T$ does not contain an antichain of size $|\mathcal P_i|$, which implies by Dilworth that $|C_T|\leq(n+1) |\mathcal P_i|$.
\end{proof}
Finally, since the intervals $[\sum_{j=1}^{i-1}|\mathcal P_j|, \sum_{j=1}^{i}|\mathcal P_j|-1]$ for $i\in[l]$ partition $[m]$, and $\mathcal F=\cup_{T\subseteq [m]}C_T$, Lemma~\ref{chainC_T} and Lemma~\ref{antichainC_T} give that $|\mathcal F|\leq|\mathcal F_b\cup\mathcal F_t|+\sum_{j=1}^l|\mathcal P_j|(n+1)=O(n)$, which finishes the proof of the theorem.

\end{proof}
As explained in the beginning of the section, the above result, together with Theorem~\ref{maintheorem}, gives the main result of this section.
\begin{theorem}\label{multipartitelinear} Let $n_1,\dots,n_k$ be positive integers. If $n_1=n_2=\dots=n_k=1$, then $\text{sat}^*(n,K_{n_1,\dots,n_k}) \break = O(1)$. Otherwise, $\text{sat}^*(n,K_{n_1,\dots,n_k}) =\Theta(n)$.
\end{theorem}
\section{Further remarks}
It is natural to ask the question of what would happen if one were to replace the antichain with a more general class of posets. In other words, what can we say about the saturation number of $\mathcal P_1*\mathcal Q*\mathcal P_2$, where $\mathcal Q$ belongs to a class of posets that strictly contains the non-trivial antichains?

It was shown in \cite{gluing} that the posets that have the UCTP (unique cover twin property) behave relatively well under linear sums, in the context of poset saturation. In particular, it was showed that if at least one of $\mathcal P_1$ and $\mathcal P_2$ has the UCTP and at least two elements, then the saturation number of $\mathcal P_1*\mathcal P_2$ is unbounded. This prompts the question: is the saturation number of $\mathcal P_1*\mathcal Q*\mathcal P_2$ unbounded for any posets $\mathcal P_1$ and $\mathcal P_2$, and any poset $\mathcal Q$ with the UCTP and at least two elements? The answer is yes.

We omit the full proof as it is a straightforward generalization of everything done in this paper. However, we provide an overview. First we observe that if $\mathcal Q$ has the UCTP and a unique maximal (minimal) element, removing that element results in a poset with the UCTP that does not have a unique maximal (minimal) element. This, coupled with Corollary~\ref{poset-point-poset-corollary} allows us to assume that $\mathcal P_1$ has a unique minimal element, $\mathcal P_2$ has a unique maximal element, and $\mathcal Q$ has neither.

The goal is to show that if $\mathcal F$ is a $\mathcal P_1*\mathcal Q*\mathcal P_2$-saturated family, then it separates the ground set. If there exists a saturated family that does not separate the ground set, in the sense that there exist $i\in[n]$ such that $A\setminus B\neq\{i\}$ for all $A,B\in\mathcal F$, then by `cloning' $i$ (as shown in \cite{freschi2023induced}), we may assume that there exists a bounded number of $i\in[n]$ such that $A\setminus B=\{i\}$ for some $A,B\in\mathcal F$. Next, we replace the condition  `$S,S\cup\{i\}\in\mathcal F$' in Lemma~\ref{firstprelim} and Lemma~\ref{secondprelim} with `there exists $A,B\in\mathcal F$ such that $A\setminus B=\{i\}$'. With this modification, and the assumption that the set $C$ in Lemma~\ref{mainprelim} is bounded, all the other proofs go through, showing that $\mathcal F$ is at least linear, a contradiction. Therefore, we have the following result.

\begin{theorem}
Let $\mathcal P_1$ and $\mathcal P_2$ be two posets, and $\mathcal Q$ be a poset with the unique cover twin property and at least two elements. Then $\text{sat}^*(n,\mathcal P_1*\mathcal A_k*\mathcal P_2)\rightarrow\infty$ as $n\rightarrow \infty$.
\end{theorem}
The final thing we would like to mention is that it was brought to our attention by Barnab\'as Janzer that in the case where one of the posets is empty, the saturation number is at least $n+1$. In other words, $\text{sat}^*(n,\mathcal P*\mathcal A_k)\geq n+1$, and $\text{sat}^*(n,\mathcal A_k*\mathcal P)\geq n+1$, for any finite poset $\mathcal P$ and integer $k\geq 2$. The main idea is to show that, given a $\mathcal P*\mathcal A_k$-saturated family, there exists for every $i$ of the ground set an element $S$ of the family such that both $S$ and $S\cup\{i\}$ are in the family. This is achieved by looking at singletons. For those that are in the family, $S=\emptyset$, and for those that are not, we look at a copy of $\mathcal P*\bullet$ in the family such that all the elements of $\mathcal P$ contain that singleton, the minimal element does not contain that singleton, and the minimal element has maximal cardinality among all such copies -- then we take $S$ to be the minimal element of such a copy of $\mathcal P*\bullet$.

\bibliographystyle{amsplain}
\bibliography{references}
\Addresses
\end{document}